\documentclass{article}

\usepackage{amsmath,amssymb,amsthm}

\usepackage{xspace}

\usepackage{setspace}

\providecommand{\abs}[1]{\ensuremath{\lvert#1\rvert}\xspace}
\providecommand{\norm}[1]{\ensuremath{\lVert#1\rVert}\xspace}
\providecommand{\ip}[1]{\ensuremath{\langle #1\rangle}\xspace}
\providecommand{\ipf}[1]{\ensuremath{\left( #1\right)}\xspace}

\newcommand{\defin}[1]{{\em #1}}
\newcommand{\C}{\ensuremath{{\mathbb C}}\xspace}
\newcommand{\R}{\ensuremath{{\mathbb R}}\xspace}
\newcommand{\N}{\ensuremath{{\mathbb N}}\xspace}
\newcommand{\Quat}{\ensuremath{{\mathbb H}}\xspace}

\newcommand{\K}{\ensuremath{{\mathbb K}}\xspace}

\newcommand{\eqsp}{\quad}

\newcommand{\lmp}{\ensuremath{\ell_2^m\rightarrow\ell_p^n}}
\newcommand{\lmpk}{\ensuremath{\ell_{2;\K}^m\rightarrow\ell_{p;\K}^n}}

\newcommand{\lpkn}{\ensuremath{\ell_{p;\K}^n}}
\newcommand{\ltkm}{\ensuremath{\ell_{2;\K}^m}}

\newcommand{\NKmp}{\ensuremath{N_{\K}(m,p)}}
\newcommand{\FKmp}{\ensuremath{\Phi_{\K}(m,p)}}

\newcommand{\matrD}{\ensuremath{{\mathcal D}}\xspace}

\newcommand{\dm}{\ensuremath{\,d}}

\theoremstyle{plain}
\newtheorem{theorem}{Theorem}

\newtheorem{lemma}[theorem]{Lemma}

\theoremstyle{definition}

\newtheorem{remark}[theorem]{Remark}

\author{Yu. I. Lyubich}
\title{Upper bound for isometric embeddings \lmp}

\date{}

\begin{document}


\bibliographystyle{amsplain}

\maketitle

\footnotetext[1]{{\em 2000 Mathematics Subject Classification:
    46B04}\newline\indent {\em Keywords:} Isometric embeddings,
  quaternion spaces }

\begin{abstract}
  The isometric embeddings $\lmpk$ ($m\geq 2$, $p\in 2\N$) over a
  field $\K\in\lbrace \R,\C,\Quat\rbrace$ are considered, and an
  upper bound for the minimal $n$ is proved. In the commutative case
  ($\K\neq\Quat$) the bound was obtained by Delbaen, Jarchow and
  Pe{\l}czy{\'n}ski (1998) in a different way.
\end{abstract}

Let $\K$ be one of three fields $\R,\C,\Quat$ (real, complex or quaternion). 
Let $\K^n$ be the $\K$-linear space consisting of columns $x=\left[\xi_i\right]_1^n$, 
$\xi_i\in\K$, with the right (for definiteness) multiplication by scalars 
$\alpha\in\K$. The normed space $\lpkn$ is $\K^n$ provided with the norm
\[
  \norm{x}_p = \left(\sum_{k=1}^n \abs{\xi_i}^p\right)^{1/p},\eqsp 1\leq p <
  \infty.
\]

For $p=2$ this space is Euclidean, $\norm{x}_2 = \sqrt{\ip{x,x}}$, where
the inner product $\ip{x,y}$ of $x$ and a vector
$y=[\eta_i]_1^n$ is
\[
  \ip{x,y} = \sum_{i=1}^n \overline{\xi_i}\eta_i.
\]

An isometric embedding $\lmpk$, $2\leq m\leq n$, may exist only if
$p\in 2\N = {2,4,6,\ldots}$, see \cite{lyubich70} for $\K=\R$ and
\cite{lyushat05_PeterMJ} for any $\K$. Conversely, under
these conditions for $m$ and $p$, there exists an $n$ such that $\ltkm$ can
be isometrically embedded into $\lpkn$, see \cite{milman88} (and also
\cite{lyuvas93,reznick92}) for $\K=\R$, \cite{konig95} for $\K=\C$, and
\cite{lyushat05_PeterMJ} for $\K=\Quat$, $\C$ and $\R$
simultaneously. The proofs of existence in these papers also yield some upper
bounds for the minimal $n=\NKmp$. According to \cite{lyushat05_PeterMJ}, 
these bounds can be joined in the inequality   
\begin{equation}
  \label{eq:1}
  \NKmp\leq \dim\FKmp,
\end{equation}
where $\FKmp$ is the space of homogeneous polynomials (forms) $\phi(x)$
over $\R$ of degree $p$ in real coordinates on $\K^m$ such that
$\phi(x\alpha)=\phi(x)$ for all $\alpha\in\K$, $\abs{\alpha}=1$. For
$\K=\R$ the latter condition is fulfilled automatically since
$p\in2\N$, so $\Phi_{\R}(m,p)$ consists of all forms of degree $p$
on $\R^m$. The space $\Phi_{\C}(m,p)$ coincides with that which was
used in \cite{konig95}. Note that in all cases
$\dim\FKmp$ can be explicitly expressed through binomial
coefficients. (All the formulas are brought together in
\cite[Theorem 2]{lyushat05_PeterMJ}.)

In the present paper we prove that
\begin{equation}
  \label{eq:2}
  \NKmp\leq \dim \FKmp-1.
\end{equation}
For $\K=\R$ and $\C$ this result (in terms of binomial coefficients)
was obtained by Delbaen, Jarchow and Pe{\l}czy{\'n}ski
\cite{delbaenetal98} as a by-product of the proof of their Theorem~B.
Their rather complicated technique essentially uses the commutativity
of the field $\K$, so it is not applicable to $\K=\Quat$. Our 
proof of (\ref{eq:2}) is general and elementary. 
Let us start with two lemmas, the first of which is well known.

\begin{lemma}
  \label{lem:1}
  A linear mapping $f:\lmpk$ is isometric if and only if there is a
  system of vectors $u_k\in\ltkm$, $1\leq k\leq n$, such that the
  identity
  \begin{equation}
    \label{eq:3}
    \sum_{k=1}^n\abs{\ip{u_k,x}}^p = \ip{x,x}^{p/2}
  \end{equation}
  holds for $x\in\ltkm$.
\end{lemma}

\begin{proof}
  A general form of $f$ as a linear mapping is
  $fx=\left[\ip{u_k,x}\right]_1^n$, where $\left(u_k\right)_1^n$ is
  a system of vectors from $\ltkm$ (called the \defin{frame} of $f$ 
\cite{lyushat05_PeterMJ,lyuvas93}).
 The identity (\ref{eq:3}) is nothing but $\norm{fx}_p=\norm{x}_2$.
\end{proof}

An isometric embedding $\lmpk$ is called \defin{minimal} if $n=\NKmp$.

\begin{lemma}
  \label{lem:2}
  If $f$ is minimal and $\left(u_k\right)_1^n$ is its frame then the 
  functions $\abs{\ip{u_k,x}}^p$ are linearly independent.
\end{lemma}

\begin{proof}
  Let
  \begin{equation}
    \label{eq:4}
    \sum_{k=1}^n\omega_k\abs{\ip{u_k,x}}^p =0
  \end{equation}
  with some real $\omega_k$, $\displaystyle\max_k \omega_k=1$, and let
  $\omega_n=1$ for definiteness. By subtraction of (\ref{eq:4}) from
  (\ref{eq:3}) we get
  \[
    \sum_{k=1}^{n-1}\left(1-\omega_k\right)\abs{\ip{u_k,x}}^p = \ip{x,x}^{p/2},
  \]
  i.e.
  \[
    \sum_{k=1}^{n-1}\abs{\ip{v_k,x}}^p = \ip{x,x}^{p/2},
  \]
  where $v_k = u_k\left(1-\omega_k\right)^{1/p}$. This contradicts the
  minimality of $f$.
\end{proof}

\begin{remark}
  Since all functions $\abs{\ip{\cdot,x}}^p$ belong to $\FKmp$, the
  inequality (\ref{eq:1}) immediately follows from Lemma \ref{lem:2}.
  However, the existence of an isometric embedding
  $\lmpk$ is assumed in this context. 
\end{remark}

Now we proceed to the proof of (\ref{eq:2}). 

\begin{proof}
  Let $f:\lmpk$ be a minimal isometric embedding. Then, according to (\ref{eq:1}), 
$n\leq\dim\FKmp$. We have to prove that the equality is impossible. 
 
Suppose to the contrary. Then the system
  $\left(\abs{\ip{u_k,x}}^p\right)_1^n$ corresponding to the frame of
  $f$ is a basis of $\FKmp$ by Lemma \ref{lem:2}. In particular,
  there is an expansion
  \begin{equation}
    \label{eq:5}
    \left(\sum_{i=1}^m\lambda_i\abs{\xi_i}^2\right)^{p/2} =
    \sum_{k=1}^n a_k(\lambda_1,\ldots,\lambda_m)\abs{\ip{u_k,x}}^p,
  \end{equation}
  where $(\lambda_i)_1^m\in\R^m$ and $a_k$ are some functions of these
  parameters.

  Now we introduce the inner product
  \[
    \ipf{\phi_1,\phi_2} = \int_S \phi_1(x)\phi_2(x)\dm\sigma(x)\eqsp
    (\phi_1,\phi_2\in\FKmp)
  \]    
  where $\sigma$ is the standard measure on the unit sphere
  $S\subset\ltkm$. In the Euclidean space $\FKmp$ we have the basis
  $\left(\theta_k(x)\right)_1^n$ dual to
  $\left(\abs{\ip{u_k,x}}^p\right)_1^n$. This allows us to represent the
  coefficients $a_k$ as
  \[
    a_k(\lambda_1,\ldots,\lambda_m) =
    \int_S\left(\sum_{i=1}^m\lambda_i\abs{\xi_i}^2\right)^{p/2} \theta_k(x)\dm\sigma(x).
  \]
  Hence, $a_k(\lambda_1,\ldots,\lambda_m)$ are forms
  of degree $p/2$, a fortiori, they are continuous. 

  Denote by $\R^m_{+}$ the open coordinate cone in $\R^m$, so 
$\R^m_{+} =\{(\lambda_i)_1^m\subset\R^m: \lambda_1>0,...,\lambda_m>0\}$. 
 We prove that \defin{on} $\R^m_{+}$ \defin{all} $a_k(\lambda_1,\ldots,\lambda_m)>0$ 
 or equivalently, $\hat{a}(\lambda_1,\ldots,\lambda_m)\equiv\displaystyle\min_k
  a_k(\lambda_1,\ldots,\lambda_m) >0$. Suppose to the contrary: let
$\hat{a}(\gamma_1,\ldots,\gamma_m)<0$ for some 
$(\gamma_i)_1^m\subset\R^m_{+}$. On the other hand, $\hat{a}(1,\ldots,1)=1$
by comparing (\ref{eq:3}) to (\ref{eq:5}) with all $\lambda_i=1$.
Since $\hat{a}$ is continuous, we have $\hat{a}(\mu_1,\ldots,\mu_m)=0$ 
for some $(\mu_i)_1^m\subset\R^m_{+}$. But
  the latter means that all $a_k(\mu_1,\ldots,\mu_m)\geq 0$ and, at
  least one of them is zero, say $a_n(\mu_1,\ldots,\mu_m)=0$. Therefore, 
  \[
    \left(\sum_{i=1}^m \mu_i\abs{\xi_i}^2\right)^{p/2} =
    \sum_{k=1}^{n-1}a_k(\mu_1,\ldots,\mu_m)\abs{\ip{u_k,x}}^p, 
  \]
  whence
  \begin{equation}
    \label{eq:6}
    \ip{z,z}^{p/2} = \sum_{k=1}^{n-1}\abs{\ip{v_k,z}}^p,
  \end{equation}
  where
  \[
    z=\matrD x, \eqsp v_k = \left(a_k(\mu_1,\ldots,\mu_m)\right)^{1/p}\matrD^{-1}u_k
  \]
  and \matrD is the diagonal matrix with entries
  $\mu_1^{1/2},\ldots,\mu_m^{1/2}$. By Lemma \ref{lem:1} the identity
  (\ref{eq:6}) means that the system $(v_k)_1^{n-1}$ is the frame of 
  an isometric embedding
  $\ltkm\rightarrow\ell_{p;\K}^{n-1}$. This contradicts the minimality of $n$. 
As a result, all $a_k(\lambda_1,\ldots,\lambda_m)\geq 0$ 
for $\lambda_1\geq 0,\ldots,\lambda_m\geq 0$, i.e. on the closed 
coordinate cone. 

  Now we take $\xi_1=1$ and $\xi_i=0$ for all $i\geq 2$, so $x=e_1$,
  the first vector from the canonical basis of $\ltkm$. In this setting 
(\ref{eq:5}) reduces to
  \[
    \lambda_1^{p/2} = \sum_{k=1}^n
    a_k(\lambda_1,\ldots,\lambda_m)\abs{\ip{u_k,e_1}}^p.
  \]
(Recall that $m\geq 2$.) This yields 
  \[
    \sum_{k=1}^n
    a_k(0,\lambda_2,\ldots,\lambda_m)\abs{\ip{u_k,e_1}}^p=0.
  \]
  
  Assume all $\ip{u_k,e_1}\neq 0$. Since for
  $\lambda_2>0,\ldots,\lambda_m>0$ all
  \mbox{$a_k(0,\lambda_2,\ldots,\lambda_m)\geq0$}, all of them are equal to zero.  
  Hence, the right side of the identity 
 (\ref{eq:5}) vanishes as long as $\lambda_1 =0$, in contrast to the function on the left  
  side, a contradiction. To finish the proof
 we only have to show that the  assumption 
$\ip{u_k,e_1}\neq 0$, $1\leq k\leq n$, is not essential.

 First, note that all $u_k\neq 0$, otherwise, the number 
$n$ in (\ref{eq:3}) would be reduced. Therefore, the sets 
$\{x: \ip{u_k,x} =0\}, 1\leq k\leq n$, are 
hyperplanes in $\ltkm$. Their union is different from the whole space. Hence, there is a 
vector $e$ such that all $\ip{u_k,e_}\neq 0$, $\norm{e}_2=1$. This $e$ can be represented 
as 
$e=ge_1$ where $g$ is an isometry of the space $\ltkm$. Indeed, this space is Euclidean, so its 
isometry group is transitive on the unit sphere. Thus, all $\ip{g^{-1}u_k,e_1}\neq 0$. 
On the other hand, $(g^{-1}u_k)_1^n$ is the frame of the isometric embedding $fg:\lmpk$.    
\end{proof}


\bibliography{all}

\begin{singlespace}
  Address:

  {\it Department of Mathematics, 

    Technion, 32000, 

    Haifa, Israel}

  \smallskip
  email: {\it lyubich@tx.technion.ac.il}  
\end{singlespace}

\end{document}